\documentclass[a4paper]{amsart}

\usepackage{tikz,xcolor,hyperref}
\makeatletter
\@namedef{subjclassname@1991}{$2020$ Mathematics Subject Classification}
\makeatother
\title{Orderings and valuations in hyperfields}

\author[Katarzyna Kuhlmann \and Alessandro Linzi \and Hanna Stoja\l owska]{Katarzyna Kuhlmann \and Alessandro Linzi \and Hanna Stoja\l owska}
\thanks{The authors would like to thank F.-V.\ Kuhlmann  for his careful reading and remarks.}
\address{Institute of Mathematics\\
University of Szczecin\\ Wielkopolska 15\\ 70-451 Szczecin,
Poland} \email{katarzyna.kuhlmann@usz.edu.pl}

\address{Institute of Mathematics\\
University of Szczecin\\ Wielkopolska 15\\ 70-451 Szczecin,
Poland} \email{linzi.alessandro@gmail.com}

\address{Institute of Mathematics\\
University of Szczecin\\ Wielkopolska 15\\ 70-451 Szczecin,
Poland} \email{hanna.stojalowska@phd.usz.edu.pl}
\date{7.06.2021}
\usepackage{geometry}                
\usepackage[applemac]{inputenc}
\usepackage[T1]{fontenc}
\usepackage{indentfirst}
\usepackage{diagmac2}
\usepackage{lipsum}
\usepackage{csquotes}
\usepackage{emptypage}
\usepackage{amsmath,amsfonts,amssymb,amsthm}
\usepackage{latexsym}
\usepackage{mathdots}
\usepackage{mathrsfs}  
\newcommand{\N}{\mathbb{N}}

\newcommand{\Q}{\mathbb{Q}}
\newcommand{\R}{\mathbb{R}}

\DeclareMathOperator{\Hom}{Hom}
\DeclareMathOperator{\sgn}{sgn}
\theoremstyle{plain}
\newtheorem{tw}{Theorem}[section]
\newtheorem{pr}[tw]{Proposition}
\newtheorem{lem}[tw]{Lemma}
\newtheorem{co}[tw]{Corollary}

\newtheorem{rem}[tw]{Remark}
\theoremstyle{definition}
\newtheorem{df}[tw]{Definition}
\newtheorem{example}[tw]{Example}

\theoremstyle{remark}
\newtheorem{re}[tw]{Remark}
\begin{document}
\begin{abstract}
We introduce and study in detail the notion of compatibility between valuations and
orderings in real hyperfields. We investigate their relation with valuations and orderings induced on factor and residue hyperfields. Much of the theory from real fields can be generalized to real hyperfields; we point out facts that cannot. We generalize the Baer-Krull theorem to real hyperfields.
\end{abstract}
\maketitle
\section{Introduction}
Hyperfields are variants of fields where the operation of addition is multivalued. They were introduced by Krasner
in 
\cite{Krasner56}. The inspiration for Krasner's hyperfields came from valued fields. The first hyperfields considered by Krasner were of the form $K^\times/(1+\mathcal M^{\gamma}) \cup \{0\}$, where $K$ is
a field with valuation $v$ and $\mathcal M^{\gamma} = \{a\in K \mid v(a)> \gamma\}$, $\gamma \in v(K^\times)$.
These hyperfields inherit the valuation from $K$, so Krasner axiomatized not only the notion of hyperfields
but also the notion of  valuations on hyperfields. 
This notion was adopted and used by Tolliver \cite{Tolliver} and Lee\cite{Lee}. In \cite{DeS} Davvaz
and Salasi introduced a more general notion of a valuation on a hyperring.\par
Meanwhile, Marshall (\cite{M}) generalized to hyperfields 
the Artin-Schreier theory of real fields, which was later further developed by Marshall and G\l adki in \cite{MG1} and applied to quadratic form theory in \cite{MG2}. \par
However, one important part of the theory of real fields was not generalized: compatibility between valuations and orderings. The lack of valuation theoretical tools kept the authors of \cite{MG2} from fully investigating the theory of higher level orderings in hyperfields. In the classical field theory the connection between valuations and orderings is established through the notion of \emph{natural valuation}. This is the finest valuation of an ordered field $(K,<)$ whose valuation ring is convex with respect to $<$. Every valuation whose valuation ring  is convex with respect to $<$ is called \emph{compatible with $<$}. The compatibility between valuations and orderings on fields, introduced independently by Baer \cite{Baer} and Krull
\cite{Krull}, laid the foundations for the development of real algebra and real algebraic geometry.  We have found it important to generalize the notion of compatibility to hyperfields because of two reasons.
The first reason is to consolidate the theory of valuations on hyperfields beyond the special case considered in 
\cite{Krasner56}, \cite{Tolliver} and \cite{Lee}. 
The second reason is to give tools to study real algebraic geometry over hyperfields (the basics of algebraic geometry over hyperfields were recently presented in \cite{Jun}) and its possible applications to the solutions of open problems in real algebra, for instance in the theory of holomorphy rings and sums of squares. 
Further, our results are intended to lay the basis for the study of holomorphy rings, spaces of compatible valuations and higher level orderings on hyperfields. \par
In Section 2 we review and systematize  basic facts of the theory of hyperrings. We distinguish between hyperrings which come from Krasner's construction (factor hyperrings) and hyperrings obtained as the quotient  of a hyperring by a hyperideal (quotient hyperrings). In the literature the notion of a subhyperring is not consistent. Some authors (\cite{Jun})
define it as a subset which with ``induced operations'' (we will clarify this in Section 2) is itself a hyperring. 
Others (\cite{DeS}) require that a subhyperring is closed under multiplication and multivalued addition. 
We address the second case by the name of \emph{strict subhyperring}.
In this section we also generalize from rings to hyperrings some facts which we will use in the next sections.\par
The third section contains basic facts from the theory of real hyperfields as in \cite{M}. Moreover, we investigate when a factor hyperfield is real (Theorem \ref{cr})
and construct an interesting example of a real hyperfield which contains a subhyperfield which is not a strict subhyperring. We believe that this example is interesting for the future development of hyperfield extension theory.
\par
In the fourth section we systematize notions and results
connected to valued hyperfields. We show that every valuation subhyperring in a hyperfield is strict. 
We use this fact to establish the connection between  valuations and valuation rings, similarly  as in the case of valued fields. 
\par
The fifth section is the most important part of
this paper. We use the results from the previous section to establish  and investigate the notion of  compatibility between valuations and orderings in hyperfields. This relation generalizes the compatibility in ordered fields.
However, this generalization is not straightforward, and there are problems which we address in detail.
Finally, we prove a version of  the Baer-Krull theorem for 
real hyperfields.

\section{Preliminaries}

Let $H$ be a nonempty set and $\mathcal{P}^* (H)$ the family of nonempty subsets of $H$. A \emph{hyperoperation} $+$ on $H$ is a function which associates to every pair $(x,y) \in H \times H$ an element of $\mathcal{P}^* (H)$, denoted by $x+y$.
For a subset $A\subseteq H$ and $x \in H$ we have $A+x: = \bigcup_{a\in A} a+x$ and $x+A = \bigcup_{a\in A} x+a$.

\begin{df} \label{hypergp}
A \emph{canonical hypergroup} is a tuple $(H,+,0)$, where $+$ is a hyperoperation on $H$ and $0$ is an element of $H$ such that the following axioms hold:
\begin{itemize}
\item[(H1)] the hyperoperation $+$ is associative, i.e., $(x+y)+z=x+(y+z)$ for all $x,y,z \in H$,
\item[(H2)] $x+y=y+x$ for all $x,y\in H$,
\item[(H3)] for every $x\in H$ there exists a unique $-x$ such that $0\in x+(-x) =: x-x$,
\item[(H4)] $z\in x+y$ implies $y\in z-x$ for all $x,y,z\in H$.
\end{itemize}
\end{df}

Some authors (see for instance \cite{Mittas}) require additionally that $x+0 = \lbrace x \rbrace$ for all $x \in H$. This axiom follows from (H3) and (H4). Indeed, suppose that $y \in x+0$ for some $x,y \in H$. Then $0 \in y-x$ by (H4). Now $y = x$ follows from the uniqueness required in (H3).

\begin{df}
A \emph{commutative hyperring} is a tuple $(R,+,\cdot,0)$ which satisfies the following axioms:
\begin{itemize}
\item[(R1)] $(R,+,0)$ is a canonical hypergroup,
\item[(R2)] $(R,\cdot)$ is a commutative semigroup and $x\cdot0= 0$ for all $x\in R$,
\item[(R3)] the operation $\cdot$ is distributive with respect to the hyperoperation $+$. That is, for all $x,y,z\in R$,
\[
x\cdot(y+z)=x\cdot y+x\cdot z.
\]

\end{itemize}
From now on, by a \emph{hyperring} we mean a commutative hyperring.
If for all $x\in R$ we have that $xy=0$ implies $x=0$ or $y=0$, then $(R,+,\cdot,0)$ is called an \emph{integral hyperdomain}.
If the operation $\cdot$ has a neutral element $1 \neq 0$, then we say that $(R,+,\cdot,0,1)$ is a \emph{hyperring with unity}.  
If  $(R,+,\cdot,0,1)$  is a hyperring with unity and 
$(R\setminus\{0\}, \cdot,1)$ is a group, then $(R,+,\cdot,0,1)$ is called a \emph{hyperfield}.
\end{df}

\begin{re} \label{dd}
The double distributivity law, i.e.,
\[
(a+b)\cdot(c+d) = a\cdot c+a\cdot d+b\cdot c+b\cdot d,
\]
does not hold in general in hyperrings and hyperfields. However, the inclusion
\begin{equation} \label{dde}
 (a+b)(c+d) \subseteq ac+ad+bc+bd
\end{equation}
holds. This is shown by Viro in \cite{Viro}.
\end{re}

Every ring is a hyperring and every field is a hyperfield in which $x+y$ is a singleton.  A method of constructing  hyperrings with mutivalued addition is given by Krasner in \cite{Krasner}.  Let $A$ be a ring with unity and $T$ a normal subgroup of its multiplicative semigroup (i.e., $xT = Tx$ for every $x \in A$). The relation 
\begin{equation} \label{factor}
x \sim y \,\,\;\text{ if and only if }\,\,\; xu = yt \text{ for some } u, t \in T
 \end{equation}
is an equivalence relation. Denote by $[x]_T$ the equivalence class of $x$ and by $A_T$ the set of all equivalence classes. The set  $A_T$ with the operations
$$ [x]_T +[y]_T  = \{[xt+yu]_T\mid \; t,u \in T\};$$
$$[x]_T \cdot[y]_T  = [xy]_T,$$
is a hyperring. If $A$ is a field, then $A_T$ is a hyperfield and for every $x \in A$, $[x]_T = xT$.
The neutral element of multivalued addition in $A_T$ is $[0]_T = \{0\}$ and the additive inverse of $ [x]_T \in A_T$ is $[-x]_T$.
In \cite{Krasner} the hyperrings constructed in the way above are called quotient hyperrings. However, to distinguish them from the hyperrings constructed in \cite{Jun}, which are direct analog of the classical quotient rings, we will call them \emph{factor hyperrings}. 

\begin{example} \label{sign}
Consider the field of real numbers $\R$ with its multiplicative subgroup $(\R^{\times})^2$. We can identify the factor hyperfield $\R_{ (\R^{\times})^2}$ with the set $H = \lbrace -1, 0, 1 \rbrace$. This hyperfield is called the \emph{sign hyperfield}.
\end{example}

\begin{df}
Let $R$ and $S$ be hyperrings. A mapping $\varphi:R\to S$ is a \emph{homomorphism of hyperrings} if it satisfies
\begin{itemize}
    \item[(HH1)] $\varphi(0_R)=0_S$;
    \item[(HH2)] $\varphi(x\cdot_R y)=\varphi(x)\cdot_S\varphi(y)$ for all $x,y\in R$;
    \item[(HH3)] $\varphi(x+_R y)\subseteq\varphi(x)+_S \varphi(y)$ for all $x,y\in R$.
\end{itemize}
A homomorphism of hyperrings $\varphi$ is said to be \emph{strict} if it satisfies the following property which is stronger than (HH3):
\begin{itemize}
\item[(HH3$'$)] $\varphi(x+_R y)=\varphi(x)+_S \varphi(y)$ for all $x,y\in R$.
\end{itemize}
A strict bijective homomorphism of hyperrings is called an \emph{isomorphism of hyperrings}. If there is an isomorphism of hyperrings $\varphi:R\to S$, then one says that $R$ and $S$ are \emph{isomorphic} and writes $R\simeq S$.
\end{df}
\begin{re}
Let $\varphi:R\to S$ be a homomorphism of hyperrings. It follows from the definition that $\varphi(-x)=-\varphi(x)$ for all $x\in R$. Indeed, for all $x\in R$ we have that $0_R\in x-x$ so that $0_S=\varphi(0_R)\in\varphi(x-x)\subseteq\varphi(x)+\varphi(-x)$. Hence, the uniqueness required in axiom (H3) implies that $\varphi(-x)=-\varphi(x)$.
\end{re}

\begin{df}\label{subh}
Let $(R,+,\cdot,0)$ be a hyperring. A subset $S$ of $R$ is a \emph{subhyperring} of $R$ if it is multiplicatively closed and with the
induced hyperaddition
\begin{align*}
    a+_S b&:=(a+b)\cap S\quad(a,b\in S)
\end{align*}
is itself a hyperring.

A subset $S$ of $R$ is a \emph{strict subhyperring} of $R$ if for all $a,b\in S$ one has that $a-b\subseteq S$ and $ab\in S$.\par
\end{df}

\begin{re}
A strict subhyperring $H$ of a hyperring $(R,+,\cdot,0)$ is a subhyperring of $R$. However, not every subhyperring  is  strict (cf.\ Example \ref{naex}).
\end{re}

\begin{df}
Let $R$ be a hyperring. A strict subhyperring $I$ of $R$ is a
\emph{hyperideal of $R$} 
if $r\cdot x \in I$, for every $r \in R$ and $x \in I$. 
\end{df}

Let $R$ be a hyperring and $I$ a hyperideal of $R$. It is proved by Jun in \cite{Jun} that the relation $\sim_I$ on $R$ given by
\begin{equation} \label{quotient}
 x\sim_I y\Longleftrightarrow (x - y) \cap I\neq \emptyset
\end{equation}
is an equivalence relation with the equivalence classes $x+I=\bigcup_{a\in I}x+a$ (see \cite[Lemma 3.3]{Jun}). The set $R/I$ of equivalence classes  with a hyperoperation defined as
\[
(x+I)+(y+I):=\{z+I \mid z\in x+y\}
\]
and a multiplication
\[
(x+I)\cdot(y+I) = xy+I
\]
is a hyperring (see \cite[Proposition 3.5]{Jun}). 
 We call $R/I$ the \emph{quotient hyperring of $R$ modulo $I$}.\par 

\begin{df}
Let $\varphi:R\to S$ be a homomorphism of hyperrings. The set 
\[
\ker\varphi:=\{x\in R\mid \varphi(x)=0_S\}
\]
is called the \emph{kernel} of $\varphi$.
\end{df}
In the case of hyperrings one recovers an analog of the classical correspondence between ideals and kernels of homomorphisms. 
\begin{pr}\label{hypker}
Let $R$ be a hyperring. The kernel of a homomoprhism of hyperrings is a hyperideal of $R$. Every hyperideal of $R$ is the kernel of some strict homomorphism of hyperrings. 
\end{pr}
\begin{proof}
Let $\varphi:R\to S$ be a  homomorphism of hyperrings. Take $x,y\in \ker\varphi$ and $z\in R$. By (HH3$'$) we have that $\varphi(x-y)\subseteq\varphi(x)-\varphi(y)=\{0_S\}$ so $x-y\subseteq\ker\varphi$. Furthermore, $\varphi(xz)=\varphi(x)\varphi(z)=0_S$. Hence, $xz\in\ker\varphi$ and $\ker\varphi$ is a hyperideal of $R$. 

Every hyperideal $I$ of $R$ is the kernel of the canonical projection $\pi:R\to R/I$ which is a strict homomorphism of hyperrings (\cite[Proposition 3.6]{Jun}). 
\end{proof}

The proofs of the next lemma and corollary are the same as the proofs in  classical ring theory.
\begin{lem}\label{easy}
If a hyperideal $I$ of a hyperring $R$ contains a unit, then $I=R$.
\end{lem}
\begin{co}\label{idfield}
The only hyperideals of a hyperfield $F$ are $\{0\}$ and $F$.
\end{co}

\begin{df}
Let $I$ be a hyperideal of a hyperring $R$.
\begin{itemize}
    \item[$(i)$] $I$ is called \emph{prime} if for all $x,y\in R$ we have that $xy\in I$ implies $x\in I$ or $y\in I$.
    \item[$(ii)$] $I$ is called \emph{maximal} if $I\subsetneq R$ and for all hyperideals $J$ of $R$ we have that $I\subsetneq J$ implies $J=R$.
\end{itemize}
\end{df}
Since every hyperideal is a strict subhyperring one can easily deduce the following.
\begin{co}
 The intersection of hyperideals of a hyperring $R$ is a hyperideal of $R$.
\end{co}

In what follows we characterize prime and maximal hyperideals by means of quotient hyperrings as in classical ring theory.
\begin{pr}\label{prmaxid}
Assume that $R$ is a hyperring with unity. Let $I$ be a hyperideal $R$.
\begin{itemize}
    \item[$(i)$] $I$ is prime if and only if $R/I$ is an integral hyperdomain.
    \item[$(ii)$] $I$ is maximal if and only if $R/I$ is a hyperfield.
\end{itemize}
\end{pr}
\begin{proof}
The proof of  $(i)$ is a copy of the corresponding proof in ring theory, since  $x+I=I$ if and only if $(x-0) \cap I \neq \emptyset$ which is equivalent to $x \in I$. \par
We come to the proof of $(ii)$. Suppose that $I$ is maximal. We have to show that $R/I$ is a hyperfield, i.e.,  every nonzero $x+I\in R/I$ has a multiplicative inverse.
Since $x+I$ is nonzero, we have that $x\notin I$. Consider the smallest hyperideal $J$ of $R$ which contains $I$ and $x$, i.e., the intersection of all hyperideals of $R$ which contain $I$ and $x$. Since $I$ is maximal we obtain that $J=R\ni 1$. By \cite[Lemma 3.22(3)]{Jun} we have
\[
1\in yx+i
\]
for some $y\in R$ and $i\in I$. By axiom (H4)  we obtain that $i\in (1-yx)$ and therefore $(1-yx)\cap I\neq\emptyset$. Thus, $1+I =yx+I$ and $y+I$ is a multiplicative inverse of $x+I$.
\par
For the converse, we assume that $R/I$ is a hyperfield and we take a hyperideal $J$ of $R$ such that $I\subsetneq J$. Take $x\in J\setminus I$. We have that $x+I\neq 0_{R/I}$ and hence there exists $y+I\in R/I$ such that
\[
xy+I=(x+I)(y+I)=1+I.
\]
Thus, $\emptyset\neq (1-xy)\cap I\subseteq (1-xy)\cap J$. Since $x\in J$ we have that $xy\in J$ and therefore $J=xy+J=1+J$. This shows that $1\in J$ which implies that $J=R$ by Lemma \ref{easy}.
\end{proof}

\section{Real hyperfields}
Based on Murray Marshall's results and terminology of \cite{M} we recall first the 
most important facts of Artin-Schreier theory that extends to hyperfields. 

Let $F$ be a hyperfield. A subset $P\subseteq F$ is an \emph{ordering} of $F$ if
\[
P+P\subseteq P,\quad P\cdot P\subseteq P,\quad P\cap -P=\emptyset,\quad P\cup-P=F^\times.
\]
If $F$ is a field, then with $P$ one can associate a relation of a linear order defined by $a<b$ iff $b-a \in P$.  In general, the ordered hyperfields do not have to be linearly ordered sets. However, the relation $a<b$ iff $b-a \subseteq P$ defines on $F$ a strict partial order, which do not have to be compatible with the hyperaddition.

A subset $T\subseteq F$ is called a \emph{preordering} in $F$ if
\[
T+T\subseteq T,\quad T\cdot T\subseteq T,\quad (F^{\times})^2 \subseteq T,\quad -1 \notin T.
\]

We say that a preordering $T \subseteq F$ is \emph{maximal} if $T \subseteq S$ for some preordering $S$ in $F$ implies that $S=T$.

The set of all orderings of a hyperfield $F$ we shall denote by $\mathcal{X}(F)$ and the set of all orderings of $F$ containing a subset $T \subset F$ we denote by $\mathcal{X}(F\mid T)$.

A hyperfield $F$ is \emph{real} if $\mathcal{X}(F)\neq\emptyset$.

The following results are proved in Marshall's article mentioned above.
\begin{tw}{\cite[Lemma 3.3 and Proposition 3.4]{M}} \label{Marshall}
Let $F$ be a hyperfield. 
\begin{enumerate}
\item Every maximal preordering $T$ of $F$ is an ordering.
 \item $F$ is real if and only if $-1 \notin \Sigma (F^{\times})^2$.
 \item For every preordering $T$ of $F$ we have 
 \[T = \bigcap_{P \in \mathcal{X}(F\mid T)} P. \]
\end{enumerate}
\end{tw}
The following remark was suggested to us by the referee.
\begin{rem}
 A hyperfield $F$ is real if  and only if there is a nontrivial homomorphism of 
 $F$ into the sign hyperfield $\mathbb S$.  Indeed, if $P$ is an ordering of $F$, then the map $\phi: F\rightarrow \mathbb S$ given by
 \[
\phi({a}):=\begin{cases}0 &\text{if }a=0\\1&\text{if }a\in P\\ -1&\text{if }a\in -P\end{cases},
\]
is a nontrivial homomorphism of hyperfields. On the other hand, if  $\phi$ is a nontrivial homomorphism of $F$ into the sign hyperfield then $\phi(-1) = -1$ and thus $-1$ cannot be a sum of squares in $F$. 
\end{rem}

\begin{lem}\label{cr0}
 Let $P$ be an ordering of a field $K$ and assume that a multiplicative subgroup $T$ of $K$ is contained in $P$. Consider the factor hyperfield $K_T = \{[a]_T\mid a\in K\}$.
 \begin{enumerate}
  \item If $a \in P$, then $[a]_T \subseteq P$.
  \item The set $P_T:=\{[c]_T \mid c\in P\}$ is an ordering of $K_T$.
 \end{enumerate}
 \end{lem}
\begin{proof}
 The first statement follows from the definition of 
 $[a]_T$ given by Equation \ref{factor}. Since $K$ is a field we have that $[a]_T = \{at\mid t\in T\}$. If $a \in P$, then $[a]_T \subseteq P$ since $T\subseteq P$.\par
 
 To prove the second statement, take $[a]_T,[b]_T \in P_T$. By (1) we have that $a , b \in P$ and since $T\subseteq P$, we obtain that $at+bs \in P$  for all $t,s\in T$. Therefore, 
\[
[a]_T+ [b]_T= \lbrace [at+bs]_T \mid t,s\in T \rbrace \subseteq P_T,
\]
showing that $P_T+ P_T \subseteq P_T$.
The proof that $P_T$ is multiplicatively closed is straightforward. 
Since $P \cup -P = K^{\times}$, we have that
$P_T \cup -P_T= K_T^\times$.
Suppose now that $[a]_T \in P_T$ and $-[a]_T = [-a]_T \in P_T$. Thus $a,-a\in P$, which is impossible. 
This shows that $P_T \cap -P_T=\emptyset$.
 \end{proof}

The following theorem gives a criterion for a factor hyperfield $K_T$ to be real.

\begin{tw}\label{cr}
Let $K$ be a field and $T$ a subgroup of $K^\times$. Then
\[
|\mathcal{X}(K_T)|=|\mathcal{X}(K|T)|.
\]
In particular, $\mathcal{X}(K_T)\neq\emptyset$ if and only if $\mathcal{X}(K|T)\neq\emptyset$.
\end{tw}
\begin{proof}
We define two injective maps $\mathcal{X}(K_T)\to\mathcal{X}(K|T)$ and $\mathcal{X}(K|T)\to\mathcal{X}(K_T)$.\par
Let $P$ be an ordering of $K$ such that $T\subseteq P$. Then $P_T:=\{[c]_T\mid c\in P\}$ is an ordering of $K_T$ by the previous lemma. We consider the map $P\mapsto P_T$ and show that it is injective. Assume that $P\neq Q$ are two orderings of $K$ which contain $T$. Let $a\in P\setminus Q$. Then $-a\in Q$ so that $-[a]_T=[-a]_T\in Q_T$ and since $Q_T$ is an ordering of $K_T$ we must have $[a]_T\in P_T\setminus Q_T$. Therefore, $P_T\neq Q_T$. We have shown that $|\mathcal{X}(K_T)|\geq |\mathcal{X}(K|T)|$.\par
For the reverse inequality, let $Q$ be an ordering of $K_T$. We take $P$ to be the union of all the equivalence classes in $K_T$ which are in $Q$. We then have $T=[1]_T\in Q$, so $T\subseteq P$. We now show that $P$ is an ordering of $K$. Pick $a,b\in P$. Then $[a]_T,[b]_T\in Q$ and since $[a+b]_T\in[a]_T+[b]_T\subseteq Q$, we obtain that $a+b\in P$. Moreover, $[ab]_T=[a]_T[b]_T\in Q$ so $ab\in P$. This shows that $P+P\subseteq P$ and $P\cdot P\subseteq P$. Take any $a\in K^\times$. If $[a]_T\in Q$, then $a\in P$. Otherwise, $[-a]_T=-[a]_T\in Q$ so that $a\in -P$. This shows that $K^\times=P\cup -P$. Finally, suppose that $a\in P\cap -P$. Then $[a]_T\in Q$ and $-[a]_T=[-a]_T\in Q$. This contradicts the fact that $Q\cap -Q=\emptyset$. Hence, $P\cap -P=\emptyset$. We have defined a map $\mathcal{X}(K_T)\to\mathcal{X}(K|T)$. In order to show that this map is injective let $Q_1\neq Q_2$ be orderings of $K_T$. For $[a]_T\in Q_1\setminus Q_2$ we have that $a\in P_1\setminus P_2$. Thus, $P_1\neq P_2$. This proves that $|\mathcal{X}(K_T)|\leq |\mathcal{X}(K|T)|$.
\end{proof}

\begin{example}
Consider the sign hyperfield $\R_{(\R^{\times})^2}$ introduced in Example \ref{sign}. The set $(\R^{\times})^2$ is the unique ordering of $\R$.  Therefore, by the previous theorem, the sign hyperfield is real with only one ordering $P_{(\R^{\times})^2} = \lbrace 1 \rbrace$ by Lemma
\ref{cr0}.
\end{example}

\begin{example} \label{naex}
Consider a non-archimedean real closed field $R$ with a natural valuation $v$ associated with its unique ordering $<$.  Let $\mathcal O_v$ be the valuation ring
of $v$ and 
let $E^+(R)$ be the set of positive units of $\mathcal O_v$. The set $E^+(R)$  is a subgroup of $R^{\times}$. Now, consider the factor hyperfield $F:=R_{E^+(R)}$. 
Observe that for $a, b \in R^{\times}$ we have
\[
[a]_{E^+(R)} = [b]_{E^+(R)} \iff \sgn(a)=\sgn(b) \wedge v(a)=v(b),
\]
where $\sgn(a)$ is a sign of $a$ in the unique ordering of $R$.
Therefore, we can identify the elements of $F^{\times}$ as
\[
[a]_{E^+(R)}=\bigl(\sgn(a),v(a)\bigr).
\]
Using the definition of multiplication in a factor hyperring we obtain
\[
\bigl(\sgn(a),v(a)\bigr)\cdot\bigl(\sgn(b),v(b)\bigr)=\bigl(\sgn(a)\cdot\sgn(b),v(a)+v(b)\bigr).
\] 
The result of multiplication by $0$ is obvious.
Let $\Gamma$ be the value group of $v$. Using the definition of the hyperaddition in a factor hyperfield, we compute 
\begin{align}\label{sgamma}
(1,\gamma_1)+(1,\gamma_2)&=\{(1,\min(\gamma_1,\gamma_2))\}&\gamma_1,\gamma_2\in\Gamma\nonumber\\
(-1,\gamma_1)+(-1,\gamma_2)&=\{(-1,\min(\gamma_1,\gamma_2))\}&\gamma_1,\gamma_2\in\Gamma\nonumber\\
(1,\gamma_1)+(-1,\gamma_2)&=\{(1,\gamma_1)\}&\gamma_1<\gamma_2\in\Gamma\\
(1,\gamma_1)+(-1,\gamma_2)&=\{(-1,\gamma_2)\}&\gamma_1>\gamma_2\in\Gamma\nonumber\\
(1,\gamma)+(-1,\gamma)&=\{(e,\delta)\mid e\in\{-1,1\},\delta\geq\gamma\}\cup\{[0]\}&\gamma\in\Gamma\nonumber
\end{align}
The last equality requires more explanation. Take $a \in R$ with $v(a) = \gamma$. We have
\[
[a]_{E^+(R)} + [-a]_{E^+(R)}=\bigl\{ [as-at]_{E^+(R)} : s, t \in E^+(R) \bigr\}. 
\]
We have that $v(s) = v(t) = 0$ and $\sgn(s) = \sgn(t) = 1$. Therefore,
$v(as-at) = v(a)+v(s-t) \geq v(a)+\min\{v(s), v(t)\} = v(a).$   Take $\delta \geq v(a) = \gamma$.
We have $\delta  = v(b)$ for some $ b\in R$. We can choose $b$ in such way that $ba^{-1}>0$
Then we have $v(ba^{-1})\geq 0$ and 
$1+ba^{-1} \in E^+(R)$. Taking $s = 1+ba^{-1}$ and $t=1$, we obtain
$v(as-at) = v(b) = \delta$. Note that $\sgn(as-at)= -\sgn(at-as)$. This finish the explanation of  the last equality 
in (\ref{sgamma}).
The additive inverse of the element $(e,\gamma)\in F^{\times}$ is  $(-e,\gamma)$. Clearly, the element $(1,0)$ is the neutral element for multiplication. Finally, the multiplicative inverse for $(e,\gamma) \in F^{\times}$ is  $(e, -\gamma)$.  \par
Since $E^+(R)$ is a subgroup of positive elements of $R$, we obtain by Theorem \ref{cr} that $F$ is a real hyperfield.  Moreover, by the construction given in the proof of Theorem \ref{cr} we obtain that $P = \lbrace (s, \gamma) \mid s = 1, \gamma\in \Gamma \rbrace$ is the unique ordering of $F$. \par
Consider the subset $S:=\lbrace (-1,0), 0, (1,0) \rbrace \subseteq F$. Since $R$ is non-archimedean ordered we have that $S\subsetneq F$. For $a,b\in S$, we have that $ab\in S$ and by easy computation one checks that $a+_S b:=(a+b)\cap S$ yields on $(S,+_S,\cdot,0)$ the structure of a hyperfield. Thus, $S$ is a subhyperring of $F$. Note that $S$ is not a strict subhyperring of $F$ since $(1,0)+(-1,0)\supsetneq S$.\par
Finally, let $H$ be the sign hyperfield. Consider the map $\varphi:H \to S$ defined as follows: $1 \mapsto (1,0), -1\mapsto (-1,0)$ and $0 \mapsto 0$. It is easy to check that $\varphi$ is a strict, bijective hyperring homomorphism. 
Thus $F$ properly contains a subhyperfield isomorphic to the sign hyperfield.\par

\end{example}
\section{Valued hyperfields}
The hyperfields investigated in 1956 by Krasner (\cite{Krasner56}) were factor hyperfields of a valued field $(K,v)$ by a subgroup of $K^\times$ of the form $1+\mathcal{M}^{\gamma}$, where $\gamma$ is a non-negative element of the value group $v(K^\times)$ and $\mathcal{M}^{\gamma} = \{a \in K \mid  v(a) > \gamma\}$. The definition of a valuation
on these particular hyperfields is given in \cite{Krasner56} and repeated  in \cite{Tolliver} and \cite{Lee}.
Since natural valuations are important  part of the theory of valued fields, it is crucial that their generalization to hyperfields is a part of the theory of valued hyperfields.  Therefore,  we adopt from \cite{DeS} the more general definition of a valued hyperfield.
\begin{df}\label{hyperval}
 Take a hyperfield $F$ and an ordered abelian group $\Gamma$ (written additively). 
 The map $v:  F \rightarrow \Gamma \cup \{ \infty \}$ is called a \emph{valuation on $F$} if it has the following properties:\\
 (V1) $v(a) = \infty $ iff $a = 0$;\\
 (V2) $v(ab) = v(a)+ v(b)$;\\
 (V3) $c \in a+b \Rightarrow v(c) \ge \min\{v(a),v(b)\}.$
\end{df}
Observe that the definition above is an obvious generalization 
of the definition of field valuation. 
\begin{co}
Take a valued field $(K,v)$  with the valuation ring $\mathcal O_v$ and let $T$ be a subgroup of the group of units $\mathcal O_v^\times$. Take the factor hyperfield $F = K_T$. Then 
the map $v_T: [a]_T\mapsto v(a)$ is  a valuation of $F$ with  value group $v(K^\times)$.
 \end{co}
\begin{proof}
 The map $v_T$ is well-defined, since $v(a) = 0$ for every $a\in T$. The properties (V1)-(V3) follow from the corresponding properties of $v$.
\end{proof}
In the case of the corollary above we say that the valuation is \emph{ induced on $F$ by $v$}.
The definition of a valuation of a hyperfield  presented  in \cite{Tolliver} and \cite{Lee}, which introduces  two additional axioms, is too restrictive as it is shown in the following example.

\begin{example}
 Take a rational function field $\R(X)$ with valuation $v$ such that $v(X) = 1$. Consider the factor hyperfield $F: = \R(X)_{\R^{\times}}$ with the valuation induced by $v$. Note that $[1]_{\R^{\times}} \neq [X-1]_{\R^{\times}}$, since
 $\frac{1}{X-1} \notin \R^{\times}$. We have that $X$ and $X-2$ are in
 $1\R^{\times} + (X-1)\R^{\times}$. Therefore, $[X]_{\R^{\times}}$, $[X-2]_{\R^{\times}} \in [1]_{\R^{\times}} + [X-1]_{\R^{\times}}$. We have 
 $v([X]_{\R^{\times}}) = 1$ and $v([X-2]_{\R^{\times}}) = 0$ which shows that 
 $v([1]_{\R^{\times}} + [X-1]_{\R^{\times}})$ is not a singleton as  required by the fourth axiom  of \cite[Definition 1.4]{Tolliver}. 
 Moreover, we have that $[1]_{\R^{\times}} + [-1]_{\R^{\times}} = \R^{\times}$ is not a ball in the ultrametric generated by $v$ as  required by the fifth axiom  of \cite[Definition 1.4]{Tolliver}.
\end{example}

\begin{example}\label{hypervalex}
Consider the factor hyperfield  $F=R_{E^+(R)}$ from Example \ref{naex}. The group $E^+(R)$ is a subgroup of $\mathcal O_v ^{\times}$ for the natural valuation $v$ of $R$, therefore $v$ induces on $F$ a valuation (denoted here also by $v$). We obtain $v(0)=\infty$ and for $(s,\gamma)\in F^{\times}$, 
\[
v((s,\gamma))=\gamma.
\]
\end{example}
\begin{lem}\label{valring}
Let $v:F\to\Gamma\cup\{\infty\}$ be a valuation on a hyperfield $F$. Then:
\begin{enumerate}
 \item $v(1)=v(-1) = 0$,
 \item $v(-a)=v(a)$ for all $a\in F$,
 \item $v(a^{-1})=-v(a)$ for all $a\in F$,
 \item if $v(a)\neq v(b)$, then for every $c \in a+b$, 
 $v(c) = \min\{v(a), v(b)\}$.
\end{enumerate}
\end{lem}
\begin{proof}
Proof of (1)-(3) is the same as in the case of field valuation. The proof of (4) can be found in \cite[Lemma 4.3]{DeS}.
\end{proof}

\begin{df}
Let $F$ be a hyperfield. A subhyperring $\mathcal{O}$ of $F$ is called a \emph{valuation hyperring} if for all $x\in F$ we have that either $x\in \mathcal{O}$ or $x^{-1}\in\mathcal{O}$.
\end{df}
From the definition above it follows that a valuation hyperring is a  commutative hyperring with unity. 
The following proposition gives an important property of valuation hyperrings. 
\begin{pr}\label{valstrict}
A valuation hyperring $\mathcal{O}$ of a hyperfield $F$ is a strict subhyperring of $F$.
\end{pr}
\begin{proof}
It suffices to show that $a-b\subseteq \mathcal{O}$ for all $a,b\in\mathcal{O}$. Take $a,b\in\mathcal{O}$ and $x\in a-b$. If $x\in \mathcal{O}$, then there is nothing to show. Otherwise, we have that $x^{-1}\in\mathcal{O}$ and thus $ax^{-1},bx^{-1}\in\mathcal{O}$. Since $x\in a-b$ we obtain from (H4) that $a\in x+b$, so, using axiom (R3),
\[
ax^{-1}\in(x+b)x^{-1}=1+bx^{-1}.
\]
We have obtained that $ax^{-1}\in (1+bx^{-1})\cap\mathcal{O}=1+_{\mathcal{O}}bx^{-1}$. By (H4) and (R3) applied to the hyperring $(\mathcal{O},+_\mathcal{O},\cdot,0)$, it follows that $xx^{-1}=1\in ax^{-1}+_\mathcal{O}(-bx^{-1})=(a+_\mathcal{O}(-b))x^{-1}$. Therefore, $x\in a+_\mathcal{O}(-b)\subseteq\mathcal{O}$. This shows that $a-b\subseteq\mathcal{O}$. 
\end{proof}
We denote by 
$\mathcal{O}^\times$ the set of units of a valuation hyperring $\mathcal O$ i.e.,
\[
\mathcal{O}^\times:=\{x\in\mathcal{O}\mid x^{-1}\in\mathcal{O}\}.
\]
\begin{lem}\label{maxideal}
Let $\mathcal{O}$ be a valuation hyperring in a hyperfield $F$. Then $\mathcal{M}:=\mathcal{O}\setminus\mathcal{O}^\times$ is the unique maximal hyperideal of $\mathcal{O}$.
\end{lem}
\begin{proof}
Take $a\in\mathcal{M}$ and $c\in \mathcal{O}$. If $ca$ is invertible in $\mathcal{O}$, then there exists $x\in\mathcal{O}$ such that $x(ca)=1$. Hence $(xc)a=1$ and $a^{-1}=xc\in\mathcal{O}$ contradicting $a\in\mathcal{M}$. This proves that $ac \in \mathcal{M}$.\par
Take $a,b\in\mathcal{M}$. We may assume that $ab^{-1}\in\mathcal{O}$ (otherwise $ba^{-1}\in\mathcal{O}$ and we can interchange the roles of $a$ and $b$). Since $\mathcal{O}$ is a strict hyperring we obtain that $1-ab^{-1}\subseteq \mathcal{O}$ and therefore, using what we just proved above, $b-a =b(1-ab^{-1})\subseteq \mathcal{M}$. \par
We have shown that $\mathcal{M}$ is a hyperideal of $\mathcal{O}$. Since $\mathcal{O}\setminus\mathcal{M}$ contains only units of $\mathcal{O}$, by Lemma \ref{easy} every proper hyperideal of $\mathcal{O}$ must be contained in $\mathcal{M}$, showing that $\mathcal{M}$ is the unique maximal hyperideal of $\mathcal{O}$.
\end{proof}
 \begin{df}
Let $F$ be a hyperfield, $\mathcal{O}$ a valuation hyperring in $F$ and $\mathcal{M}$ its unique maximal hyperideal. By Proposition \ref{prmaxid}, the quotient hyperring $\overline{F}:=\mathcal{O} / \mathcal{M}$ is a hyperfield, called the \emph{residue hyperfield}.
\end{df}

The following lemma, which describes the relation between valuation subhyperrings and their maximal ideals is well known for valuation rings.
Since the proof for valuation subhyperrings is exactly the same, we will skip it.

\begin{lem}\label{incl}
Take a hyperfield $F$ and two valuation hyperrings $\mathcal{O}_1$, $\mathcal{O}_2$ in $F$. Let $\mathcal{M}_i$ be the maximal hyperideal of $\mathcal{O}_i$ for $i=1,2$. Then
\[
\mathcal{O}_1\subseteq\mathcal{O}_2\Longleftrightarrow\mathcal{M}_2\subseteq\mathcal{M}_1.
\]
\end{lem}

\begin{pr}
Let $v$ be a valuation on a hyperfield $F$. Then $\mathcal O_v: = \lbrace x \in F \mid v(x) \geq 0 \rbrace$ is a valuation hyperring of $F$ and $\mathcal M_v := \lbrace x \in F \mid v(x) > 0 \rbrace$ its unique maximal hyperideal.
\end{pr}
\begin{proof}
We first prove that $\mathcal{O}_v$ is a subhyperring of $F$. Take $a,b\in \mathcal{O}_v$. Then for all $c\in a-b$ we have $v(c)\geq\min\{v(a),v(-b)\}=\min\{v(a),v(b)\}\geq 0$ so $a-b\subseteq \mathcal{O}_v$. From the condition (V2) of Definition \ref{hyperval} it follows that $ab\in\mathcal{O}_v$. By Lemma \ref{valring} we conclude that if $x\notin\mathcal{O}_v$, then $x^{-1}\in\mathcal{O}_v$, so $\mathcal{O}_v$ is a valuation hyperring in $F$.\par
Next we show that $\mathcal{M}_v$ is the unique maximal hyperideal of $\mathcal{O}_v$. Observe that, by virtue of Lemma \ref{valring}, 
\[
\mathcal{O}_v^\times=\{x\in\mathcal{O}_v\mid v(x)=0\}.
\]
Hence $\mathcal{M}_v=\mathcal{O}_v\setminus\mathcal{O}_v^\times$ and thus $\mathcal{M}_v$ is the unique maximal hyperideal of $\mathcal{O}_v$ by Lemma \ref{maxideal}.
\end{proof}

To conclude this section we show that, as in classical valuation theory, any valuation hyperring of a hyperfield $F$ induces a valuation on $F$.

\begin{pr}
Let $F$ be a hyperfield and $\mathcal{O}$ a valuation hyperring in $F$. Let $\mathcal{O}^\times$ be a group of units of $\mathcal{O}$.  Consider the group $\Gamma:=F^\times/\mathcal{O}^\times$ and define a relation $\leq$ on $\Gamma$ as follows
\[
a\mathcal{O}^\times\leq b\mathcal{O}^\times\iff ba^{-1}\in\mathcal{O}.
\]
Then $(\Gamma,\cdot,\leq)$ is an ordered abelian group and the canonical projection $\pi:F\to\Gamma\cup\{\infty\}$, extended so that $\pi(0)=\infty$, is a valuation on $F$. 
\end{pr}
\begin{proof}
First we show that $\leq$ is a  linear order on $\Gamma$ compatible with multiplication. Since $aa^{-1}=1\in\mathcal{O}$, reflexivity is clear. If $ab^{-1},ba^{-1}\in\mathcal{O}$, then $ab^{-1}\in\mathcal{O}^\times$ so $a\mathcal{O}^\times=b\mathcal{O}^\times$. Hence $\leq$ is antisymmetric. If $ab^{-1},bc^{-1}\in\mathcal{O}$, then $ac^{-1}=ab^{-1}bc^{-1}\in\mathcal{O}$, showing that $\leq$ is transitive. Take now $a,b\in F^\times$ such that $a\mathcal{O}^\times\leq b\mathcal{O}^\times$ and $c\in F^\times$. We have that $bc(ac)^{-1}=bcc^{-1}a^{-1}=ba^{-1}\in\mathcal{O}$, whence $ac\mathcal{O}^\times\leq bc\mathcal{O}^\times$. This shows that $\leq$ is compatible with the operation of $\Gamma$. Finally, $\leq$ is a total order since $\mathcal{O}$ is a valuation hyperring, so $ab^{-1}\in\mathcal{O}$ or $ba^{-1}\in \mathcal{O}$ for all $a,b\in F^\times$.\par
We now show that $\pi$ is a valuation on $F$. Clearly, $\pi$ is a surjective map, onto the ordered abelian group $\Gamma$ with $\infty$ and (V1) holds. 
Since $\pi$ is a homomorphism of groups we obtain (V2). It remains to show that (V3) holds for $\pi$. Take $x,y\in F$. If one of them is $0$, then (V3) is obvious. We may then assume that $x,y\in F^\times$ and that $x\mathcal{O}^\times\leq y\mathcal{O}^\times$. Take $z\in x+y$. We wish to show that $zx^{-1}\in\mathcal{O}$. By assumption we have that $yx^{-1}\in\mathcal{O}$, thus
\[
zx^{-1}\in(x+y)x^{-1}=1+yx^{-1}\subseteq\mathcal{O},
\]
since $\mathcal O$ is strict by  Proposition \ref{valstrict}. This completes the proof.
\end{proof}

\section{Compatibility between orderings and valuations in hyperfields}
An ordering $<$ of a field $K$ is \emph{compatible} with a valuation $v$ if the valuation ring $\mathcal O_v$ of $v$ is convex with respect to $<$. The order relation of a real hyperfield $F$, defined by 
$$ a<b \Leftrightarrow b-a \subseteq P.$$
for an ordering $P$ of $F$, allows us to consider the notion of convexity in  $F$. However, we will see in Proposition \ref{comconv}, that it is not enough to define a compatibility between ordering and valuation in this manner.

Every real field $K$ have characteristic 0, so it contains the rationals. If $P$ is an ordering of $K$, then the set
\begin{align*}
A(P)&:=\{a\in K\mid n\pm a \in P \text{ for some } n \in \N \}
\end{align*}
is a valuation ring of $K$ (associated with the natural valuation of $P$) with the maximal ideal
\begin{align*}
I(P)&:=\{a\in K\mid \frac{1}{n}\pm a \in P\text{ for all } n \in \N \}.
\end{align*}
For details we refer the reader to Section 2 of \cite{Lam}.

The real hyperfields may not even contain all natural numbers.
The simplest example is given by the sign hyperfield in which
$1+1 =\{1\}$.
However, since every ordering is additively closed and contains 1, we have that
\[
0\notin \underbrace{1+\ldots+1}_{n\text{ times}}.
\]

Let $F$ be a hyperfield. For $n\in \N$, define
\[
I_n:=\underbrace{1+\ldots+1}_{n\text{ times}}.
\]

\begin{lem}\label{Inm}
With the notations introduced above we have, for all $n,m\in\N$
\[
I_n+I_m=I_{n+m}\quad\text{and}\quad I_n\cdot I_m\subseteq I_{nm}.
\]
\end{lem}
\begin{proof}
The first equality is just a consequence of the associativity of  addition. The second follows from
Remark \ref{dd}.
\end{proof}
\begin{example}
The equality $I_n\cdot I_m=I_{nm}$ in general might not hold. Consider the factor hyperfield $\Q_{(\Q^{\times})^2}$. 
From the well known equality
\[
(a^2+b^2)(c^2+d^2)=(ad+bc)^2+(ac-bd)^2
\]
it follows that in this hyperfield 
$I_2\cdot I_2=I_2\subsetneq I_4$ since there are 
positive rationals which are sums of four squares but not of two squares.
\end{example}
Take a hyperfield $F$ with an ordering $P$.
Define 
\begin{align*}
A(P)&:=\{a\in F\mid (I_n\pm a)\cap P\neq\emptyset \text{ for some } n\in \N \}\\
I(P)&:=\{a\in F\mid 1\pm I_n\cdot a\subseteq P  \text{ for all } n\in \N\}.
\end{align*}

\begin{pr}
The set $A(P)$ is a valuation hyperring in $F$ with unique maximal hyperideal $I(P)$. Moreover, $I_n \subseteq A(P)$ for every $n\in \N$.
\end{pr}
\begin{proof}
 First,  observe that from the definition of $A(P)$ it follows  that if $a\in A(P)$, then also $-a\in A(P)$. Take $a, b \in A(P)$. Thus there are $n, m \in \N$
 such that $(I_n\pm a)\cap P\neq\emptyset$ and $(I_m\pm b)\cap P\neq\emptyset$.  Take $c\in a+b$. Assume first that $c\in P$. Then
 $I_{n+m}+c \subseteq P$. Take  $s\in (I_n-a)\cap P$ and $t\in(I_m-b)\cap P$. By axiom (H4) we have that $a\in I_n-s$ and $b\in I_m-t$. Hence, using Lemma \ref{Inm}, we obtain:
\[
c\in a+b\subseteq (I_n-s)+(I_m-t)=I_{n+m}-(s+t).
\]
Therefore, there exists $u\in s+t\subseteq P$ such that $c\in I_{n+m}-u$. Using again (H4) we obtain that $u\in (I_{n+m}-c)$, so
$(I_{n+m}-c) \cap P\neq\emptyset$ and thus $c \in A(P)$.
If $c \notin P$, then $-c \in (-a-b) \cap P$ and as above we show that 
$-c \in A(P)$, so also $c\in A(P)$. We have shown that $a+b \subseteq A(P)$.

We will show that $ab\in A(P)$ when we take $a,b\in A(P)$. According to our first observation in this proof we may assume that $a,b\in P$. Then for every $k\in\N$, $I_k+ab\subseteq P$.  Like before, let $n,m$ be natural numbers such that $(a-I_n)\cap P\neq \emptyset$ and $(b-I_m)\cap P\neq\emptyset$. Take $s\in (I_n-a)\cap P$ and $t\in (I_m-b)\cap P$. Pick $p\in I_n$ and $q\in I_m$ such that $s\in p-a$ and $t\in q-b$. Then,  $p\in s+a$ and $q\in t+b$ by axiom (H4). Using the inclusion \ref{dde} we obtain:
\[
pq\in(s+a)(t+b) \subseteq \underbrace{st+sb+at}_{\subseteq P}+ab.
\]
Therefore, $pq\in x+ab$ for some $x\in P$ which implies that $x\in pq-ab\subseteq I_n\cdot I_m-ab\subseteq I_{nm}-ab$. Hence $(I_{nm}\pm ab)\cap P\neq\emptyset$ and therefore
$ab\in A(P)$. 
We have proven that $A(P)$ is a strict subhyperring~of~$F$. \par
For the next part of the proof we will need that $A(P)$ is a ring with unity. Indeed, we have
$$1 = 1+0\in 1+(1-1) = (1+1) - 1 = I_2  - 1,$$
thus $(I_2-1) \cap P \neq \emptyset$.
Since $I_2+1\subseteq P$ we have that 
 $(I_2\pm 1) \cap P \neq \emptyset$ and thus $1 \in A(P)$. \par
Now we will show that $A(P)$ is a valuation hyperring of $F$. 
Take  $a\in F\setminus A(P)$. Then from the observation above, $a \neq 1$ and according to the first observation of this proof we may assume that  $a\in P$. Then also $a^{-1} = a(a^{-1})^2 \in P$ and thus $1+a^{-1}\subseteq P$. Since $a\notin A(P)$ but $1+a \subseteq P$ we obtain that  $1-a\subseteq -P \cup \{0\}$. In fact
$1-a\subseteq -P $, because $a\neq 1$.
Thus
 $1 - a^{-1} = a^{-1}(a-1) \subseteq P$ and therefore   $a^{-1}\in A(P)$. \par
 
Now we will prove that $I(P)$ is the unique maximal hyperideal of $A(P)$. 
Take $ a \in A(P)$ which is not a unit. As before, we may assume that $a\in P$. Therefore,  $1+ I_na\subseteq P$ for every $n\in \N$ and hence also
$I_n +a^{-1}\subseteq P$ for every $n\in \N$. 
Since $a^{-1} \notin A(P)$  we must have that $(I_n -a^{-1})\cap P = \emptyset$ for all $n\in \N$. Therefore, $a^{-1}-I_n \subseteq P\cup\{0\}$. Since $A(P)$ is a strict subhyperring we have that $I_n \subset A(P)$. Therefore, $a^{-1}\notin I_n$  and thus
 $a^{-1}-I_n \subseteq P$. Multiplying by
 $a$ we obtain that $1-I_na \subseteq P$ for every $n\in \N$.
 This proves that $a \in I(P)$.
 \end{proof}

\begin{example}
Consider the quotient hyperfield $F=\Q_{(\Q^\times)^2}$ and let  $P :=\Q^{+}_{(\Q^\times)^2}$ be an ordering induced on this hyperfield by the unique ordering $\Q^+$ of $\Q$. Since every positive rational is a sum of at most four squares, we have that $I_4 = P$. Since
$I_4 \subset A(P)$ we have that $A(P) =F$.
\end{example}
\begin{df}
 An ordering $P$ of a hyperfield $F$ is called \emph{archimedean} if $A(P) = F$.
\end{df}

\begin{example}
Consider the factor hyperfield $F=R_{E^+(R)}$ from Example \ref{naex} with its unique ordering 
\[
P = \lbrace (s, \gamma) \in F^{\times} \mid s=1 \rbrace.
\]
Observe that $(1,0)+(1,0) = \lbrace (1,0) \rbrace$, so $I_n = \lbrace (1,0) \rbrace$ for every $n \in \N$. Hence $A(P) = \lbrace a \in F \mid ((1,0) \pm a) \cap P \neq \emptyset \rbrace$. \par
Clearly, $0 \in A(P)$. Take $a \in F^{\times}$. Then  $a = (s, \gamma)$, where $s \in \lbrace -1, 1 \rbrace$ and $\gamma \in \Gamma$. If $s = 1$, then $(1,0) + (s, \gamma) \subseteq P$. On the other hand, we have that the set $(1,0)-(s,\gamma) = (1,0) + (-s, \gamma)$ has a nonempty intersection with $P$ if and only if $\gamma \geq 0$. Hence,
\[
A(P) = \lbrace (s, \gamma) \in F^ \times \mid \gamma \geq 0, s \in \lbrace -1,1 \rbrace \rbrace \cup \lbrace 0 \rbrace \subsetneq F,
\]
which means that $P$ is a non-archimedean ordering.
Since $A(P)^\times = \lbrace (s, 0)  \mid  s \in \lbrace -1,1 \rbrace \rbrace $, the maximal ideal of $A(P)$ is the set
\[
I(P) = \lbrace (s, \gamma) \in F^ \times \mid \gamma > 0, s \in \lbrace -1,1 \rbrace \rbrace \cup \lbrace 0 \rbrace.
\]
Note that $A(P)=\mathcal{O}_v$ and $I(P)=\mathcal{M}_v$ where $v$ is the valuation described in  Example \ref{hypervalex}.
\end{example}

\begin{pr} \label{4eq}
Let $F$ be a hyperfield with a valuation $v$ and an ordering $P$.
The following conditions are equivalent:
\begin{itemize}
    \item[$(i)$] $A(P)\subseteq\mathcal{O}_v$,
    \item[$(ii)$] $\overline{P} := \lbrace a + \mathcal M_v \mid a \in P \cap {\mathcal O_v}^{\times} \rbrace$ is an ordering of $\mathcal O_v/\mathcal{M}_v$,
    \item[$(iii)$] $1+\mathcal{M}_v\subseteq P$,
    \item[$(iv)$] if $(b+a)\cap P\neq \emptyset$ and $(b-a)\cap P\neq \emptyset$, then $v(a)\geq v(b)$. 
\end{itemize}
\end{pr}
\begin{proof}
To show that $(i)$ implies $(ii)$, we first prove that if $A(P)\subseteq\mathcal{O}_v$ and $
a + \mathcal M_v = b+ \mathcal M_v$  where $ b \in P \cap {\mathcal O_v}^{\times} $, then also $a \in P \cap {\mathcal O_v}^{\times}$. Since $b$ is a unit than also $a$ must be a unit in $\mathcal O_v$.    From the definition of $I(P)$ we obtain immediately that $1+I(P) \subseteq P$. 
Using Lemma \ref{incl} we obtain that $ab^{-1}+\mathcal M_v = 1+\mathcal{M}_v\subseteq 1+I(P)\subseteq P$. In particular, $ab^{-1} \in P$ and therefore $a = ab^{-1} \cdot b \in P$. \par

Now we wish to show that $\overline P$ is a ordering of $\mathcal O_v/\mathcal{M}_v$. Take $a + \mathcal{M}_v$, $b + \mathcal{M}_v \in \overline{P}$. We have $(a + \mathcal{M}_v) + (b + \mathcal{M}_v) = \lbrace c + \mathcal M_v \mid c \in a + b \rbrace$. Since $a+b \subseteq P$ we have that $c \in P$ for every $c \in a+b$. If there is $c \in a+b$ which is not a unit in ${\mathcal O_v}$, then  $(a+b) \cap \mathcal M_v \neq \emptyset$, which means that $a+\mathcal M_v = -b + \mathcal M_v$. By our first observation, this is impossible since $a, b \in P$. Hence $\overline P + \overline P \subseteq \overline P$. Moreover, $(a + \mathcal{M}_v) \cdot (b + \mathcal{M}_v) = ab + \mathcal{M}_v \in \overline{P}$, since $ab \in P \cap {\mathcal O_v}^{\times}$. 
Take $c \in {\mathcal O_v}$. Then $c \in P$ or $-c \in P$. If $c \in \mathcal M_v$, then $c + \mathcal M_v = 
0+  \mathcal M_v$.  If $c\in P \cap {\mathcal O_v}^{\times}$, then $c + \mathcal M_v \in \overline{P}$ and if $c\in -P \cap {\mathcal O_v}^{\times}$, then $c + \mathcal M_v \in -\overline{P}$. Since $P \cap -P = \emptyset$, therefore also  $\overline{P} \cap -\overline{P} = \emptyset$.
\par
To prove the implication from $(ii)$ to $(iii)$, we assume that $\overline{P}$ is an ordering of $\mathcal O_v/\mathcal{M}_v$. Take any $x \in 1 + \mathcal M_v$. Since the cosets are equivalence classes of an equivalence relation (see \cite[Lemma 3.3]{Jun}), we have that
$x +\mathcal M_v = 1 + \mathcal M_v$. Therefore 
$x$ is a unit in $\mathcal O_v$ and $x +\mathcal M_v \in \overline{P}$. If $-x\in P$, then 
$-(x +\mathcal M_v) = -x +\mathcal M_v \in \overline{P}$
and thus $\overline{P}\cap -\overline{P} \neq \emptyset$
which contradicts that $\overline{P}$ is an ordering. Therefore $x\in P$ and thus $1+ \mathcal M_v \subseteq P$. \par

Now we show that $(iii)$ implies $(iv)$. Assume that $(iii)$ holds and that $(b\pm a)\cap P\neq\emptyset$. Suppose that $v(a)<v(b)$ so that $v(ba^{-1})>0$ and thus $\pm ba^{-1}\in\mathcal{M}_v$. By $(iii)$ we have that $1\pm ba^{-1}\subseteq P$. Let us distinguish two cases: if $a\in P$, then multiplying by $a$ we obtain $(a\pm b)\subseteq P$. However, this implies that $(a-b)\subseteq P$ so that $(b-a)\cap P=\emptyset$, a contradiction. If $a\notin P$, then $-a\in P$ and multiplying by $-a$ we obtain $(-a\pm b)\subseteq P$ so that $(-a-b)\subseteq P$ and thus $(a+b)\cap P=\emptyset$ which contradicts again our assumption. \par

Finally, we prove that $(iv)$ implies $(i)$. Assume that $(iv)$ holds and take $a\in A(P)$. 
Then there exist $n\in \N$ and $x\in I_n$ such that $(x-a)\cap P\neq\emptyset$.
Since $\mathcal O_v$ is closed under additive inverses we may without loss of generality assume that $a \in P$.  Thus $x+a\subseteq P$, so by $(iv)$ we obtain that $v(a)\geq v(x)\geq 0$. Therefore, $a\in\mathcal{O}_v$.  This completes the proof.
\end{proof}

\begin{df}
A valuation $v$ on a hyperfield $F$ satisfying the equivalent  conditions of Proposition \ref{4eq} is called   \emph{a valuation compatible with the ordering $P$}.
We denote by $\mathcal X(F,v)$ the set of all orderings of $F$ compatible with $v$. 
The ordering $\overline P$ is called \emph{the ordering induced by $P$ on the residue field $\mathcal O_v/\mathcal{M}_v$}.
The valuation associated with the valuation ring $A(P)$ is called \emph{the natural valuation associated with $P$}. 
\end{df}

For a hyperfield $F$ with a valuation $v$ and $a\in\mathcal{O}_v$, denote  $$\overline a: = a+\mathcal{M}_v.$$ The residue hyperfield $\mathcal{O}_v / \mathcal{M}_v$ denote by  $\overline F$. 

\begin{pr}
Let $v$ be the natural valuation associated with an ordering $P$ on the real
hyperfield $F$.  The ordering $\overline P$ induced by $P$ on the residue hyperfield $\overline F = A(P)/I(P)$ is archimedean.
\end{pr}
\begin{proof}
First observe that \[ \overline {I_n}:=\underbrace{\overline 1+\ldots+\overline 1 }_{n\text{ times}} = \{\overline x\mid x \in I_n\}.
\]
 Now we will show that if $a \in A(P)$ and $x\in (I_n+a)\cap P$ or $x\in (I_n-a)\cap P$, then $x+1 \subseteq P \cap A(P)^\times$. Since $x \in P$ and $P$ is additively closed, we have that $x+1 \subseteq P$.
 Take $s \in x+1$. Then $x \in s-1$, so $s-1\cap P \neq \emptyset$. Since $s \in P$, we have $s+1 \subseteq P$.  By (iv) of Proposition \ref{4eq} we have that $0 = v(1) \geq v(s)$. Since $s\in A(P)$ we have that $v(s)=0$, so $s \in A(P)^\times$. \par
 
 Finally take $\overline a \in \overline F$. Since $a \in A(P)$, there is $n \in \N$ and $x,y \in I_n$ such that $x \in (I_n+a)\cap P$ and $y \in (I_n-a)\cap P$.
 Take $s \in x+1$ and $t \in y+1$. Then  $s \in (I_{n+1}+a)\cap P$ and $y \in (I_{n+1}-a)\cap P$. Moreover, $s, t \in A(P)^\times \cap P$ by what we have shown before.
 Therefore, $\overline s \in (\overline{I_{n+1}} +\overline a )\cap \overline P$ and
 $\overline t \in (\overline{I_{n+1}} -\overline a )\cap \overline P$, which proves that $\overline a \in  A(\overline P)$. Hence $A(\overline P) = \overline F$, showing that $\overline P$ is archimedean.
  \end{proof}
  
  We will now investigate the relation between the notion of compatibility of a  valuation $v$  with an ordering of  a real hyperfield and the notion of convexity of the valuation ring of $v$. We will start with the following lemma.
  \begin{lem}
   Let $v$ be a valuation on the real hyperfield $F$, compatible with an ordering $P$. Then $\mathcal O_v$  is convex with  respect to $P$, i.e.,
   if $a, b \in \mathcal O_v$ and $a<x<b$, then $x \in \mathcal O_v$.
  \end{lem}
\begin{proof}
 Recall that the relation $<$ is given by $a<b \Leftrightarrow b-a \subset P$.  Assume that $a, b \in \mathcal O_v$ and $a<x<b$.
 
 If $x \in P$, then also $b \in P$ and thus
 $b+x \subseteq P$. Since also $b-x \subseteq P$, by 
 Proposition \ref{4eq}, $v(x)\geq v(b)\geq 0$, so 
 $x  \in \mathcal O_v$.
 
 If $-x \in P$, then also $-a \in P$ and thus
 $-a-x \subseteq P$. Since  $x-a  = -a+x \subseteq P$, by 
 Proposition \ref{4eq}, $v(x)\geq v(-a) = v(a)\geq 0$, so again
 $x \in \mathcal O_v$.
\end{proof}

In the next proposition we show that the convexity of a valuation ring is not enough to obtain compatibility between valuation and ordering.
  
  \begin{pr}\label{comconv}
   Let $v$ be a valuation on a real field $K$ and let $P$ be an ordering of $K$.
   Take $T = P\cap \mathcal O_v^\times$ and consider the factor hyperfield $K_T$ with valuation $v_T$ induced by $v$ and ordering $P_T$ induced by $P$. Then:
\\    (i)  The valuation $v$ is compatible with $P$ if and only if $v_T$ is  compatible with $P_T$.\\
(ii) If $v$ is a rank 1 valuation which is not compatible with $P$, then any two elements in $K_T$ which are positive (with respect to $P_T$)  are not comparable. In particular, $\mathcal O_{v_T}$ is convex with respect to $P_T$, while $v_T$ is not compatible with $P_T$.
    \end{pr}
\begin{proof}
 (i)  Observe that since $T\subset \mathcal O^\times_v$, $[x]_T\in \mathcal M_{v_T}$ if and only if $[x]_T\subset \mathcal M_{v}$ if and only if $x \in \mathcal M_v$.\par
 Assume that $v$ is compatible with $P$. Then $1+\mathcal M_v \subseteq P$, by Proposition \ref{4eq}.
 Since $T \subseteq P$, we have that $1h_1+xh_2 = h_1(1+xh_2h_1^{-1} ) \in P$, 
 for every $h_1, h_2 \in T$ and $x\in \mathcal M_v$. Therefore,   $[1]_T+[x]_T = \{[1h_1+xh_2]_T \mid h_1, h_2 \in T\} \subseteq P_T$ for every
  $x \in \mathcal M_v$, which proves that $v_T$ is compatible with $P_T$. \par
  Now assume that $[1]_T+[x]_T \subseteq P_T$ for every $x \in \mathcal M_v$.
  Then $[1+x]_T \in P_T$ and thus $1+x \in P$ which proves that $v$ is compatible with $P$.\par
  (ii) Recall that the order relation induced by $P_T$ is  given by
  $$[a]_T < [b]_T \Leftrightarrow [b]_T-[a]_T \subseteq P_T.$$
  This relation is compatible with multiplication by positive elements. Indeed,
  if $[b]_T-[a]_T \subseteq P_T$ and $[c]_T \in P_T$ then $[c]_T([b]_T-[a]_T) = 
  [c]_T[b]_T-[c]_T[a]_T \subseteq P_T$. Therefore, if
  $[0]_T<[a]_T < [b]_T$, then $[0]_T<[ab^{-1}]_T < [1]_T$.
  Hence, to prove that any two positive elements of $K_T$ are incomparable, it is enough to prove that no positive element of $K_T$ is comparable with $[1]_T$,
  i.e., to prove that for every $x\in P$ such that $[x]_T \neq [1]_T$,
  \[ ([x]_T - [1]_T) \cap P_T \neq \emptyset \text{ and } ([1]_T - [x]_T )\cap P_T \neq \emptyset.\]
  Observe that $[x]_T = [1]_T$ if and only if $x\in T$, if and only if $x$ is positive and $v(x) = 0$. Thus we assume that $v(x)\neq 0$ and first we consider 
  the case where $v(x)>0$. Here we have to consider two possibilities:
  \begin{itemize}  
   \item $x-1 \in P$.\\
   Then $[x-1]_T \in ([x]_T - [1]_T )\cap P_T $. Since $v(x)>0$, we have that $v(x+1) = 0$ and thus $x+1 \in T$. Therefore, we have
   $[1]_T =[1\cdot(1+x) -x\cdot 1]_T \in  ([1]_T -[x]_T )\cap P_T $. 
   
   \item $1-x \in P$.\\
   Then $[1-x]_T \in ([1]_T - [x]_T) \cap P_T $. 
   Since $P$ is not compatible with $v$, we have that $1+\mathcal M_v \nsubseteq P$, so there is a positive element $s \in \mathcal M_v$ such that $s-1 \in P$. Since $v$ is a valuation  of rank 1, its value group is archimedean. Thus there is a natural number $n$ such that 
   $v(s^n) = nv(s) >v(x). $ Then $v(\frac{s^n}{x})>0$, hence $v(\frac{x+s^n}{x})=
   v(1+\frac{s^n}{x})  =0$. By positivity of $x$ and $s$ we obtain that $\frac{x+s^n}{x} \in T$. Also $1+x \in T$. 
   We compute $x(\frac{x+s^n}{x}) -1 (1+x) = s^n -1$.
   Since $s>1$, we have that $s^n >1$, hence $s^n -1 \in P$. Therefore, $[s^n -1]_T \in  ([x]_T - [1]_T) \cap P_T $. 
   
  \end{itemize}

Now we will consider the case when $v(x)<0$. Since the order relation associated with $P_T$ is compatible with multiplication by positive elements it follows that if $[x]_T$ would be comparable with $[1]_T$, then also $[x]_T^{-1} = [x^{-1}]_T$
would be comparable with  $[1]_T$. Since $v(x^{-1})>0$, this is impossible as we have just shown.\par

Since any two elements of $K_T$ with the same sign with respect to $P_T$ are incomparable, then it follows that every subring of $K_T$ is convex with respect to $P$.
Therefore, $\mathcal O_{v_T}$ is convex with respect to $P_T$, but $v_T$ is not compatible with $P_T$ by (i).
\end{proof}

We say that a valuation is \emph{real} if its residue hyperfield is real.

\begin{pr}\label{exists}
Let $v$ be a real valuation on a hyperfield $F$. Take an ordering $\mathfrak p$ in $\overline F$. Then there exists an ordering $P$ in $F$ such that $P$ is compatible with $v$ and $\overline P = \mathfrak p$.
\end{pr}

\begin{proof}
Consider the set 
\[ 
T = \lbrace a \in F \mid  \overline{ax^2} \in \mathfrak p \text{ for some }  x \in F   \rbrace.
\]
We will show that $T$ is a pre-ordering of $F$. Take $a, b \in T$. Then there exist $x,y \in F$ such that $\overline{ax^2}, \overline{by^2} \in \mathfrak p$. We see that $\overline{ab(xy)^2} = \overline{ax^2} \cdot \overline{by^2} \ \in \mathfrak p$, so $ab \in T$. Next we observe that for every $c \in F^{\times }$ we have  $\overline{c^2 \cdot (c^{-1})^2} = \overline 1 \in \mathfrak p$, so $F^{ \times 2} \subseteq T$. If $-1 \in T$, then $-(\overline x)^2 = \overline{-x^2} \in \mathfrak p$ for some $x \in F$, a contradiction since $\mathfrak p$ contains all non-zero squares of $\overline{F}$. To show $T + T \subseteq T$ take $a,b \in T$. This means that there exist $x,y \in F$ such that $\overline{ax^2}, \overline{by^2} \in \mathfrak p$. Observe that then $ax^2, by^2 \in \mathcal O_v^\times$. Without loss of generality we may assume that $xy^{-1} \in \mathcal O_v$ (if not, then $x^{-1}y \in \mathcal O_v$ and we may exchange the roles of $x$ and $y$ in the following argument). Take $c \in a+b$. Then $cx^2 \in (a+b)x^2 = ax^2(1+ba^{-1})$. Observe that $ba^{-1} = (ax^{2})^{-1} \cdot by^2 \cdot (xy^{-1})^2$. Since $(ax^2)^{-1}, by^2, (xy^{-1})^2 \in \mathcal O_v$, we obtain $\overline{ba^{-1}} = \overline{(ax^2)^{-1}} \cdot \overline{by^2}\cdot \overline{(xy^{-1})^2} \in \mathfrak p \cup \lbrace 0 \rbrace$ and $\overline{1 + ba^{-1}} = \overline 1 + \overline{ba^{-1}} \subseteq \mathfrak p$. Hence $\overline{cx^2} \in \overline{ax^2} \cdot \overline{1+ba^{-1}} \subseteq \mathfrak p$ and therefore $c \in T$, so $T + T \subseteq T$.\par
Applying Zorn's Lemma we  obtain that every preordering of a hyperfield $F$ is contained in some maximal preordering, which  by Lemma \ref{Marshall} is an ordering. Therefore, $T \subseteq P$ for some ordering 
$P$  of $F$. We will show that $\overline P=\mathfrak p$. Take $\overline a \in \overline P$. Then $a\in P\cap\mathcal{O}_v^\times$. Suppose that $\overline a\notin\mathfrak p$. Then $-\overline{ a }= \overline {-a\cdot 1^2 }\in \mathfrak p$ which implies that $-a \in T \subseteq P$. This contradiction shows that $\overline P\subseteq \mathfrak p$. If $\overline a\in\mathfrak p$, then we have $a\in\mathcal{O}_v^\times$ and therefore $a=a1^2\in T\subseteq P$. Thus $\overline a\in\overline P$. 
\end{proof}

Exactly like for orderings in fields one introduces the notion of a \emph{signature} of an ordering $P$ of a hyperfield $F$. It is a group homomorphism \[ \sgn_P:F^{\times}\rightarrow \{1,-1\}\]
defined by
\[
\sgn_P({a}):=\begin{cases}1&\text{if }a\in P\\ -1&\text{if }a\in -P\end{cases}.
\]
The ordering $P$ is the kernel of $\sgn_P$.

\begin{tw}[Baer-Krull Theorem for real hyperfields]
Let $v$ be a valuation on a real hyperfield $F$ with the value group $\Gamma$. Then there is a bijection
between the set $\mathcal X(F,v)$ and the set $\mathcal X(\overline F) \times \Hom(\Gamma,\{1,-1\})$.
\end{tw}
\begin{proof}
By Proposition \ref{exists}, for every $\mathfrak p \in \mathcal X(\overline F)$  we can choose $P_{\mathfrak p}\in\mathcal{X}(F,v)$ which  induces $\mathfrak p$ on $\overline F$. 
For an ordering $P \in \mathcal{X}(F,v)$ which induces ${\mathfrak p}$ on $\overline F$ we define a map $\chi_P: \Gamma\rightarrow \{1,-1\}$ by 
$$\chi_P(v(a)) = \sgn_{P_{\mathfrak p}}(a)\sgn_P(a).$$
Note that $\chi_P$ is well defined. Indeed,
if $v(a) = v(b)$, then $ab^{-1}$ is a unit in $\mathcal O_v$, which means that $\overline{ab^{-1}} \in \overline F^\times$.
Since both $P_{\mathfrak p}$ and $P$ induce ${\mathfrak p}$ on $\overline F$, we have that $\sgn_P({ab^{-1}})= \sgn_{\mathfrak p}(\overline{{ab^{-1}}}) =\sgn_{P_{\mathfrak p}}({ab^{-1}})  $ and therefore, $\chi_P(v(a)) = \chi_P(v(b))$. 
Moreover, $\chi_P$ is a homomorphism since
$\chi_P(v(a)+v(b)) = \chi_P(v(ab)) = \sgn_{P_{\mathfrak p}}(ab)\sgn_P(ab) = \sgn_{P_{\mathfrak p}}(a)\sgn_P(a)\sgn_{P_{\mathfrak p}}(b)\sgn_P(b) = \chi_P(v(a))\chi_P(v(b)).$ \par

Define the map $\Phi: \mathcal X(F,v) \rightarrow \mathcal X(\overline F) \times \Hom(\Gamma,\{1,-1\})$ by $$\Phi(P) = (\overline P, \chi_P).$$ To show that $\Phi$ is a bijection we will show that for every ordering $ \mathfrak p \in \mathcal X(\overline F)$ and every character $h \in \Hom(\Gamma,\{1,-1\})$ there is a unique ordering $P \in \mathcal X(F,v)$ such that $\Phi(P) = (\mathfrak p, h)$.\par

Define
\[
P:=\{x\in F^\times \mid (h(v(x))=1\wedge x\in P_{\mathfrak p})\vee(h(v(x))= -1\wedge -x\in P_{\mathfrak p})\}.
\]
First we will prove that $P$ is an ordering  of $F$. Take $a,b\in P$. We  distinguish four cases:\\
CASE 1. If $a,b\in P_{\mathfrak p}$, then $h(v(a))=h(v(b))=1$ and so $h(v(ab))=h(v(a)+v(b))=h(v(a))h(v(b))=1$. From $ab\in P_{\mathfrak p}$ it now follows that $ab\in P$.

Without loss of generality  we assume that $v(a)\leq v(b)$.
Take $c\in a+b\subseteq P_{\mathfrak p}$.
Then $b\in c-a$, hence $(c-a)\cap P_{\mathfrak p}\neq\emptyset$.
Also $c+a\subseteq P_{\mathfrak p}$ and therefore by $(iv)$ of Proposition \ref{4eq} we obtain $v(a)\geq v(c)$.  On the other hand,  $v(c)\geq \min\{v(a),v(b)\} = v(a)$. Hence $v(c)=v(a)$ and thus $h(v(c))=h(v(a))=1$. This shows that $a+b\subseteq P$.\\
CASE 2. If $a\in P_{\mathfrak p}$ but $b\notin P_{\mathfrak p}$, then $h(v(a))=1$ and $h(v(b))=-1$. We have that $ab\in -P_{\mathfrak p}$ and $h(v(ab))=h(v(a)+v(b))=h(v(a))h(v(b))=-1$, so $ab\in P$. 

Take now $c\in a+b$. If $c\in P_{\mathfrak p}$, then on the one hand $(a+b)\cap P_{\mathfrak p}\neq\emptyset$, on the other hand $a-b\subseteq P_{\mathfrak p}$. By $(iv)$ of Proposition \ref{4eq}  we obtain that $v(a)\leq v(b)$.
Since $v(a)\neq v(b)$, we have that
 $v(a)<v(b)$. Hence $v(c)=v(a)$.
Then $h(v(c))=h(v(a))=1$ and therefore, $c\in P$.
Assume now that $-c\in P_{\mathfrak p}$. Hence $(-a-b)\cap P_{\mathfrak p}\neq\emptyset$ and $a-b\subseteq P_{\mathfrak p}$. Again, by $(iv)$ of Proposition \ref{4eq}, we obtain that $v(a)>v(-b)=v(b)$.  Thus, $v(c)=v(b)$ and therefore, $h(v(c))=h(v(b))=-1$, whence $c\in P$. This shows that $a+b\subseteq P$ also in this case.\\
CASE 3. The case where $a\notin P_{\mathfrak p}$ and $b\in P_{\mathfrak p}$ is symmetrical to CASE 2.\\
CASE 4. The case where $a,b\notin P_{\mathfrak p}$ is analogous to  CASE 1. \\
The four cases above show that $P$ is additively and multiplicatively closed.
It remains to show that $P\cap-P=\emptyset$ and that $P\cup -P=F^\times$.  Since  $v(x)=v(-x)$ and   $P_{\mathfrak p}$ is an ordering it follows that $x\in P$ if and only if $-x \notin P$, which proves both equations. \par

From the definition of $P$ it follows that for every $x \in F^\times$,
$\sgn_P(x)\sgn_{P_\mathfrak p}(x) = h(v(x))$, so $h =\chi_P$.  \par
Finally, assume that $Q$ is an ordering of  $F$ inducing $ \mathfrak p$ on $\overline F$ such that $\chi_Q = h = \chi_P$. Then $\sgn_Q = \sgn_P$ and therefore $Q=P$.
\end{proof}

\end{document}